\documentclass{article}
\usepackage{amsmath,amssymb,hyperref,theorem}
\usepackage[french]{babel}

\newcommand{\nobracket}{}
\newcommand{\nocomma}{}
\newcommand{\nosymbol}{}
\newcommand{\tmdummy}{$\mbox{}$}
\newcommand{\tmem}[1]{{\em #1\/}}
\newcommand{\tmname}[1]{\textsc{#1}}
\newcommand{\tmop}[1]{\ensuremath{\operatorname{#1}}}
\newcommand{\tmstrong}[1]{\textbf{#1}}
\newcommand{\tmtextbf}[1]{{\bfseries{#1}}}
\newcommand{\tmtextrm}[1]{{\rmfamily{#1}}}
\newenvironment{proof}{\noindent\textbf{Preuve\ }}{\hspace*{\fill}$\Box$\medskip}
\newtheorem{definition}{D\'efinition}
\newtheorem{lemma}{Lemme}
{\theorembodyfont{\rmfamily}\newtheorem{remark}{Remarque}}
\newtheorem{theorem}{Th\'eor\`eme}

\begin{document}

\title{R{\'e}solution du $\partial \bar{\partial}$ pour les courants
prolongeables d{\'e}finis sur un domaine pseudoconvexe non born{\'e} de
$\mathbb{C}^n$}
\author{ Eramane Bodian, Waly Ndiaye et Salomon Sambou}
\maketitle

\begin{abstract}
  On r{\'e}sout le $\partial \bar{\partial}$ pour les courants prolongeables
  d{\'e}finis dans un domaine pseudoconvexe non born\'e de $ \mathbb{C}^{n} $.
  
  \ \ \ \ \ \ \ \ \ \ \ \ \ \ \ \ \ \ \ \ \ \ \ \ \ \ \ \ \ \ \ \ \ \ \ \ \ \
  \ {\textbf{Abstract}}
  
  We solve the $\partial \bar{\partial}$-problem for extensible currents
  defined on a unbounded pseudoconvex domain of $ \mathbb{C}^{n} $.
  
  {\textbf{Mots clefs}} : Courant prolongeable , $\partial \bar{\partial}$ ,
  cohomologie de De Rham.
  
  {\textsl{ math{\'e}matique }} 2010 : 32F32.
\end{abstract}

{\tmstrong{Introduction}}

Soit $\Omega \subset \mathbb{C}^n$ un ouvert, on se pose la question suivante
:

Si $T$ est un courant $d$-ferm{\'e} sur $\Omega$, existe - t - il un courant
prolongeable sur $\Omega$ tel que $\partial \bar{\partial} u = T$ ? Ayant
r{\'e}pondu {\`a} la question pour le cas o{\`u} $\Omega$ est la boule
euclidienne de $\mathbb{C}^n$ dans {\cite{SBD}} et aussi le cas des domaines
born{\'e}s et leurs compl{\'e}mentaires dans {\cite{HBS}} et {\cite{BDS}} , on
se demande si on peut r{\'e}pondre {\`a} la question pour tout domaine non
born{\'e} de $\mathbb{C}^n$ dont le compl{\'e}mentaire est aussi non
born{\'e}.

Tenant compte de consid{\'e}rations classiques, nous devons pour r{\'e}pondre
{\`a} cette question, avoir {\`a} r{\'e}soudre une {\'e}quation $( \ast)
\text{} d u = T$, o{\`u} $T$ est un courant prolongeable, la solution obtenue
se d{\'e}compose alors en une partie $\partial$-ferm{\'e}e et l'autre \
$\bar{\partial}$-ferm{\'e}e. Il est n{\'e}cessaire d'avoir des conditions
g{\'e}om{\'e}triques sur $\Omega$ pour obtenir des solutions du $\partial$
respectivement du $\bar{\partial}$ pour les courants prolongeables. Le domaine
$\Omega$ v{\'e}rifie comme dans \ {\cite{SBD}}, {\cite{HBS}} et {\cite{BDS}},
que les groupes de cohomologie de De Rham $H^j ( \Omega)$ et $H^j ( b \Omega)$
sont nuls pour $j \geqslant 1$. Le fait que $\Omega$ soit non born{\'e} ainsi
que son compl{\'e}mentaire nous oblige {\`a} adopter la r{\'e}solution {\`a}
support exact de l'op{\'e}rateur $d$. La r{\'e}solution du $\partial
\bar{\partial}$ devient alors une cons{\'e}quence des r{\'e}sultats de
r{\'e}solution du $\bar{\partial}$ pour les courants prolongeables obtenus
dans {\cite{Jud}}.

\section{Pr{\'e}liminaires et notations}

Soit $X$ est une vari{\'e}t{\'e} complexe de dimension $n$. On note
$\check{\mathcal{D}}'^p  (\Omega)$ l'espace des $p$-courants d{\'e}finis sur
$\Omega$ et prolongeables {\`a} X, $A_c^p ( \bar{\Omega})$ les $p$-formes
diff{\'e}rentielles de classe $\mathcal{C}^{\infty}$ sur $X$ {\`a} support
compact dans $\bar{\Omega}$. Sur $X$, on note $\check{\mathcal{D}}'^{p, q} 
(\Omega)$ l'espace des $(p, q)$-courants prolongeables d{\'e}finis sur
$\Omega$ et $A_c^{p, q} ( \bar{\Omega})$ l'espace des $(p, q)$-formes
diff{\'e}rentielles {\`a} support compact dans $\bar{\Omega}$. On note
$\check{\mathrm{H}}^p (\Omega)$ le $p^{\mathrm{ieme}}$ groupe de cohomologie
de De Rham des courants prolongeables d{\'e}finis sur $\Omega$,
$\check{\mathrm{H}}^{p, q} (\Omega)$ le $(p, q)^{\mathrm{ieme}}$ groupe de
cohomologie de Dolbeault des courants prolongeables d{\'e}finis sur $\Omega$.
Si $F \subset X$, alors \tmtextrm{H}$_{\infty}^p (F)$ d{\'e}signe le
$p^{\mathrm{ieme}}$ groupe de cohomologie de De Rham des $p$-formes
diff{\'e}rentielles de classe $\mathcal{C}^{\infty}$ d{\'e}finis sur $X$,
\tmtextrm{H}$_{\infty, c}^p (X)$ est le groupe de cohomologie de De Rham des
$p$-formes diff{\'e}rentiables de classe $\mathcal{C}^{\infty}$ sur $X$ {\`a}
support compact et enfin $A^p (F)$ l'espace des $p$-formes diff{\'e}rentielles
de classe $\mathcal{C}^{\infty}$ sur $F$. On note aussi, pour tout domaine $D$
de $X$, $b D$ le bord de $D$.

\begin{definition}
  Un courant $T$ d{\'e}fini sur $\Omega$ est dit prolongeable s'il existe un
  courant $\check{T}$ d{\'e}fini sur $\mathbb{C}^n$ tel que $\check{T}_{|
  \Omega} = T$.
\end{definition}

\section{R{\'e}solution de l'{\'e}quation $du = T$}

On consid{\`e}re $X =\mathbb{R}^{n + 1}$ et
\[ \Omega = \{ x \in \mathbb{R}^{n + 1} / x_{n + 1} > 0 \} \subset
   \mathbb{R}^{n + 1} \]
un domaine contractile et $b \Omega =\mathbb{R}^n \times \{ 0 \} \simeq
\mathbb{R}^n$ et son compl{\'e}mentaire
\[ \mathfrak{C}=\mathbb{R}^{n + 1} \backslash \bar{\Omega} =\mathbb{R}^n
   \times \{ x_{n + 1} < 0 \} . \]
$\Omega$ est un convexe non born{\'e} et son compl{\'e}mentaire $\mathfrak{C}$
est aussi convexe et non born{\'e}. On a donc $H^j ( \Omega) = 0$ \ et $H^j (
b \Omega) = 0$ pour $j \geqslant 1$. Ainsi le r{\'e}sultat principal de cette
partie est le suivant :

\begin{theorem}
  \label{thm1}{\tmdummy}
  
  \[ \check{\mathrm{H}}^j (\Omega) = 0 \text{\tmop{pour}} 1 \leq j \leq n + 1.
  \]
\end{theorem}

Pour d{\'e}montrer le th{\'e}or{\`e}me \ref{thm1}, on a besoin du lemme
suivant~:

\begin{lemma}
  \label{lem1}$A_c^p ( \bar{\Omega}) \cap \ker d = d ( A_c^{p - 1} (
  \bar{\Omega}))$ pour $1 \leq p \leq n + 1$.
\end{lemma}

\begin{proof}

  Soit $f \in A_c^p ( \bar{\Omega}) \cap \ker d \nocomma$ , alors il existe
  $\Omega'$ une boule de centre $z_0$ et de rayon $R$ telle que pour $f \in
  A_c^p (\Omega') \cap \ker d \nocomma$ , $0 < p \leqslant n + 1$ , il existe
  $g \in A_c^{p - 1} ( \Omega')$ avec $d g = f$. Cela entraine que $d g_{| B
  \nobracket} = 0$ o{\`u} $B = \Omega' \cap ( \mathbb{R}^{n + 1} \backslash
  \Omega')$. Si $p = 1$ , alors $g$ est une constante {\`a} support compact
  donc $g = 0$ sur $B$.
  
  Si $1 < p \leqslant n + 1$ , on a $g_{| B \nobracket}$ est une $(p -
  1)$-forme $d$-ferm{\'e}e d'o{\`u} il existe donc une $(p - 2)$-forme
  diff{\'e}rentielle $h$ de classe $\mathcal{C}^{\infty}$ sur $\bar{B}
  \nobracket$ telle que $dh = g_{| B \nobracket}$. Soit $\tilde{h}$ une
  extension $\mathcal{C}^{\infty}$ {\`a} support compact de $h$ {\`a}
  $\Omega'$ (on peut utiliser l'op{\'e}rateur d'extension de {\cite{See}}) ,
  $u = g - d \tilde{h}$ est une $(p - 1)$-forme diff{\'e}rentiable de classe
  $\mathcal{C}^{\infty}$ sur $\mathbb{R}^{n + 1} \backslash \Omega$ {\`a}
  support compact dans $\Omega$ et $du = f$.
\end{proof}

{\tmstrong{Preuve}} (th{\'e}or{\`e}me \ref{thm1})

D'apr{\`e}s Martineau {\cite{Mart}}, Puisque $\dot{\bar{\Omega}} = \Omega$ ,
les courants d{\'e}finis sur $\Omega$ et prolongeables {\`a} $\mathbb{R}^{n +
1}$ sont des {\'e}l{\'e}ments de $( D^r ( \bar{\Omega}))'$ dual des $r$-formes
diff{\'e}rentielles de classes $\mathcal{C}^{\infty}$ sur $\mathbb{R}^{n +
1}$ {\`a} support compact sur $\bar{\Omega}$. Cependant $\bar{\Omega}$
n'{\'e}tant pas born{\'e} , $D^r ( \bar{\Omega})$ est une limite inductive
d'espaces de Fr{\'e}chet.

On va consid{\'e}rer un compact $K \subset \bar{\Omega}$ de $\mathbb{R}^{n +
1}$ et $D^r ( K)$ l'espace des $r$-formes diff{\'e}rentielles sur
$\mathbb{R}^{n + 1}$ {\`a} support compact dans $K$.
\begin{eqnarray*}
  L_T^K : d ( D^r ( \Omega) \cap D^r ( K) \cap \ker d) \longrightarrow &
  \mathbb{C} & \\
  \bar{\partial} \varphi  \longmapsto & \langle T, \varphi \rangle & 
\end{eqnarray*}
lin{\'e}aire continue et $d L_T^K$ s'{\'e}tend {\`a} un op{\'e}rateur
lin{\'e}aire continue

$\widetilde{L^K_T}$ : $D^{r + 1} ( \bar{\Omega}) \cap D^{r + 1} ( K)
\longmapsto \mathbb{C}$ \ c'est un courant prolongeable et

$d \widetilde{L^K_T} = ( - 1)^{n - r + 1} T$ sur $\dot{K}$.

On va consid{\'e}rer une famille $( K_n)_{n \in \mathbb{N}}$ de compacts de
$\bar{\Omega}$ on a sur $K_n$ , $\exists$ $S_n$ un courant prolongeable tel
que $d S_n = T$ sur $\dot{K_n}$ on a $K_n \Subset \dot{K}_{n + 1}$.

$S_{n + 1} - S_n$ est $d$-ferm{\'e} et $S_{n + 1} - S_n = d v_n$ sur
$\dot{K}_{n + 1}$.

Soit $\chi = 1$ sur un voisinage de $K_n$ contenu dans $K_{n + 1}$ et
\[ S_{n + 1} - d ( \chi v_n) = S_n + d ( 1 - \chi) v_n  \text{\tmop{sur}}
   \dot{K_n} \]
posons $U_{n + 1} = S_{n + 1} - d ( \chi v_n)$ et $U_n = S_n + d ( 1 - \chi)
v_n$.

On a $d U_{n + 1} = d U_n = T$ sur $\dot{K}_n$ et $U_{n + 1} = U_n$ sur
$K_n$. On pose
\[ S = \lim_n U_{n + 1} \]
c'est un courant prolongeable sur $\Omega$ et v{\'e}rifie $d S = T$.

\section{R{\'e}solution du $\partial \bar{\partial}$ pour les courants
prolongeables}

On va consid{\'e}rer maintenant le cas
\[ \Omega = \{ z = ( z_1, \ldots, z_n) \in \mathbb{C}^n / \tmop{Im} z_n > 0
   \} . \]

On donne le r{\'e}sultat suivant de r{\'e}solution du $\bar{\partial}$ {\`a}
support exact :

\begin{theorem}
  \label{thm5}
\end{theorem}

Soit $\Omega$ un domaine et $f \in A_c^{p, q} ( \bar{\Omega}) \cap \ker
\bar{\partial}$. Alors il existe $g \in A_c^{p, q - 1} ( \bar{\Omega})$ telle
que $\bar{\partial} g = f$ ; $1 \leqslant q \leqslant n$.

\begin{proof}

  C'est une cons{\'e}quence du r{\'e}sultat de r{\'e}solution du
  $\bar{\partial}$ {\`a} support exact de
  
  {\cite[th{\'e}or{\`e}me 4.2]{Jud}}. Si le support de $f$ est compact dans
  $\Omega$ , alors on prend un domaine pseudoconvexe $\Omega'$ dans $\Omega$
  qui contient le support de $f$ d'apr{\`e}s le r{\'e}sultat de
  {\cite[th{\'e}or{\`e}me 4.2]{Jud}} , il existe $g \in A_c^{p, q - 1} (
  \overline{\Omega'})$ telle que $\bar{\partial} g = f$.
  
  Si maintenant $\tmop{supp} ( f) \cap b \Omega \neq \varnothing$ , puis que
  $f$ est {\`a} support compact et $b \Omega$ L{\'e}vi plat , on peut trouver
  $K \subset \bar{\Omega}$ un compact d'int{\'e}rieur pseudoconvexe {\`a} bord
  lisse qui contient le support de $f$. D'apr{\`e}s {\cite[th{\'e}or{\`e}me
  4.2]{Jud}} \ il existe $h$ une $( p, q - 1)$-forme diff{\'e}rentielle {\`a}
  support dans $K$ telle que $d h = f \nosymbol .$ On {\'e}tend $h$ par $0$
  sur $\mathbb{C}^n \backslash K$ et on a la solution recherch{\'e}e. Ainsi
  pour toute $f \in A_c^{p, q} ( \bar{\Omega}) \cap \ker \bar{\partial}$, il
  existe $g \in A_c^{p, q - 1} ( \bar{\Omega})$ telle que $\bar{\partial} g =
  f$.
\end{proof}

Par dualit{\'e} classique (cf th{\'e}or{\`e}me \ref{thm1}), on a le
r{\'e}sultat suivant :

\begin{theorem}
  \label{thm5}
\end{theorem}

Soit $\Omega = \{ z = ( z_1, \ldots, z_n) \in \mathbb{C}^n / \tmop{Im} z_n >
0 \}$ et $T$ un courant prolongeable $\bar{\partial}$-ferm{\'e} sur $\Omega$.
Alors il existe $S$ un courant prolongeable d{\'e}fini sur $\Omega$ tel que
$\bar{\partial} S = T$.

Nous pouvons donc {\'e}tablir le r{\'e}sultat suivant :

\begin{theorem}
  
\end{theorem}

Soit $\Omega = \{ z = ( z_1, \ldots, z_n) \in \mathbb{C}^n / \tmop{Im} z_n >
0 \} \subset \mathbb{C}^n$ un domaine, alors pour tout $T$ un $( p, q)$
courant d{\'e}fini sur $\Omega$, prolongeable et $d$-ferm{\'e} , il existe $S$
un $( p + q - 1)$ courant sur $\Omega$ et prolongeable tel que $\partial
\bar{\partial} S = T$ avec $1 \leqslant p, q \leqslant n - 1$.

\begin{proof}

  Soit $T$ un $(p, q)$-courant, $1 \leqslant p \leqslant n$ et $1 \leqslant q
  \leqslant n$, $d$-ferm{\'e} d{\'e}fini sur $\Omega$ et prolongeable avec $1
  \leqslant p + q \leqslant 2 n$. Puisque le th{\'e}or{\`e}me \ref{thm1} \
  nous assure que
  
  $\check{\mathrm{H}}^{p + q} (\Omega) = 0$, il existe un courant
  prolongeable $\mu$ d{\'e}fini sur $\Omega$ tel que $d \mu = T$. $\mu$ est un
  $(p + q - 1)$-courant, il se d{\'e}compose en un $(p - 1, q)$-courant
  $\mu_1$ et en un
  
  $(p, q - 1)$-courant $\mu_2$. On a
  \[ d \mu = d (\mu_1 + \mu_2) = d \mu_1 + d \mu_2 = T. \]
  Comme $d = \partial + \bar{\partial}$, on a pour des raisons de bidegr{\'e},
  $\partial \mu_2 = 0$ et $\bar{\partial} \mu_1 = 0$. On obtient $\mu_1 =
  \partial u_1$ et $\mu_2 = \bar{\partial} u_2$ avec $u_1$ et $u_2$ des
  courants prolongeables d{\'e}finis sur $\Omega$ d'apr{\`e}s le
  th{\'e}or{\`e}me \ref{thm5}. On a donc :
  \begin{eqnarray*}
    T & = & \partial \mu_2 + \bar{\partial} \mu_1\\
    & = & \partial \bar{\partial} u_2 + \bar{\partial} \partial u_1\\
    & = & \partial \bar{\partial}  (u_2 - u_1)
  \end{eqnarray*}
  Posons $S = u_2 - u_1$ , $S$ est un $(p - 1, q - 1)$-courant prolongeable
  d{\'e}fini sur $\Omega$ tel que $\partial \bar{\partial} S = T$.
\end{proof}

\begin{remark}
  On a la d{\'e}composition de la solution sous la forme suivante :
  \[ S = S^{( p - 1, q)} + S^{( p, q - 1)} + S^{( p - 2, q + 1)} + \cdots +
     S^{( 0, p + q - 1)} + S^{( p + q - 1, 0)} \]
  mais l'autre partie, pour des raisons de bidegr{\'e} est $d$-ferm{\'e}e.
  Quitte {\`a} prendre une autre solution il y'a pas perte de
  g{\'e}n{\'e}ralit{\'e} de consid{\'e}rer $S = S^{( p - 1, q)} + S^{( p, q -
  1)}$.
\end{remark}

Universit{\'e} Assane Seck de Ziguinchor

{\tmname{Mamadou Eramane Bodian}} \
{\tmem{\href{https://fr-mg42.mail.yahoo.com/neo/launch\#}{\label{yui_3_16_0_ym19_1_1480498582910_3453}eramane20era@yahoo}}}

{\tmname{Waly Ndiaye}} :
{\tmem{\href{https://fr-mg42.mail.yahoo.com/neo/launch\#}{\label{yui_3_16_0_ym19_1_1480498582910_2223}walyunivzig@yahoo.fr}}}

{\tmname{Salomon Sambou}} :
{\tmem{\href{https://fr-mg42.mail.yahoo.com/neo/launch\#}{\label{yui_3_16_0_ym19_1_1480498582910_3711}ssambou@univ-zig.sn}}}

\end{document}